\documentclass[11pt]{amsart}
\usepackage{geometry}
\usepackage{amsmath, amsthm, amscd, amsfonts, amssymb, graphicx, xcolor,bm}
\usepackage{mathrsfs}
\usepackage{hyperref}
\usepackage{mathtools}
\usepackage{enumitem}
\usepackage{xcolor}

\newtheorem{theorem}{Theorem}
\newtheorem{lemma}{Lemma}
\newtheorem{proposition}{Proposition}
\newtheorem{corollary}{Corollary}

\theoremstyle{definition}

\theoremstyle{remark}
\newtheorem{remark}{Remark}
\newtheorem*{ack}{\bf Acknowledgements}
\numberwithin{equation}{section}

\begin{document}
	\setcounter{page}{1}

	\title[Necessary And Sufficient Conditions for Absolute Monotonicity]{Necessary and Sufficient Conditions for Absolute Monotonicity of Functions Related to Gaussian Hypergeometric Functions}

	\author[Tiehong Zhao]{Tiehong Zhao}
	
	\address{Tiehong Zhao, School of mathematics, Hangzhou Normal University, Hangzhou 311121, Zhejiang, China}
	\email{\textcolor[rgb]{0.00,0.00,0.84}{tiehong.zhao@hznu.edu.cn}}
	\urladdr{https://orcid.org/0000-0002-6394-1049}

	
	
	
	
	\thanks{$^{\dag}$Corresponding Author: Tiehong \textsc{Zhao}}
	
	\begin{abstract}
		This paper systematically investigates the absolute monotonicity of two function families associated with the Gaussian hypergeometric function $F(a, b; c; x)$ (where $a,b,c\in\mathbb{R}_+$): $\mathcal{F}_p(x)=(1-x)^pF(a,b;c;x)$ and $\mathcal{G}_p(x)=(1-x)^p \exp(F(a,b;c;x))$, as well as the logarithmic transform $\ln\mathcal{F}_p(x)$.
		Our primary goal is to establish necessary and sufficient conditions for the parameter $p$ such that $-\mathcal{F}'_p$, $\pm\mathcal{G}'_p$ and $\pm(\ln\mathcal{F}_p)'$ are absolutely monotonic on $(0,1)$. Additionally, we derive several results regarding the absolute monotonicity of their higher-order derivatives. As applications, we derive several new inequalities for the Gaussian hypergeometric function $F(a,b;c;x)$. Most importantly, we develop a novel constructive approach based on Jurkat's criterion for power series ratios, which avoids limitations of cumbersome recursive/inductive methods in existing literature.
	\end{abstract} \maketitle
	
	\section{Introduction}
	
Throughout this paper, we denote by $\mathbb{N}$ (resp. $\mathbb{R}_+$) the set of positive integers (resp. positive real numbers), and let $\mathbb{N}_0=\mathbb{N}\cup\{0\}$. 	For $x,y\in\mathbb{R}_+$, we define the Euler gamma function $\Gamma(x)$, beta function $B(x,y)$, psi function $\psi(x)$, Euler-Mascherioni constant $\gamma$, and Ramanujan $R$-function $R(a,b)$ \cite{CCHC-OM-2022} as follows:
\begin{align*}
\Gamma(x)&=\int_{0}^\infty t^{x-1}e^{-t}dt,\qquad B(x,y)=\int_{0}^{1}t^{x-1}(1-t)^{y-1}=\frac{\Gamma(x)\Gamma(y)}{\Gamma(x+y)},\\
\psi(x)&=\frac{\Gamma'(x)}{\Gamma(x)},\hspace{2.25cm} \gamma=\lim_{n\to\infty}\left(\sum_{k=1}^{n}\frac{1}{k}-\ln n\right)=0.5772\cdots
\end{align*}
and
\begin{equation*}
	R(a,b)=-2\gamma-\psi(a)-\psi(b)
\end{equation*}with a know special case $ R\left(\frac{1}{2},\frac{1}{2}\right)=4\ln2$.

For complex numbers $a,b$ and $c$ (where $-c\notin\mathbb{N}_0$), the Gaussian hypergeometric function (denoted as ${_2F_1(a,b;c;x)}$ in standard notation)  is defined on the interval $(0,1)$ by the power series \cite{AVV-NY-1997}
\begin{equation*}
	F(a,b;c;x)=\sum_{n=0}^{\infty}\frac{(a)_n(b)_n}{(c)_n}\frac{x^n}{n!},
\end{equation*}
where $(a)_n=\Gamma(n+a)/\Gamma(a)$ is the shifted factorial function (also known as Pochhammer symbol) for $n\in\mathbb{N}$. Here,  (with $\mathrm{Re}(x)>0$) is the gamma function. A notable subclass of Gaussian hypergeometric functions arises when $c=a+b$, such functions are called {\it zero-balanced} hypergeometric functions.  The behavior of $F(a,b;c;x)$ near $x=1^-$ (i.e., as $x$ approaches $1$ from the left) is well-characterized by three cases:
\begin{itemize}[leftmargin=1.6em]
	\item $c>a+b$ (c.f. \cite[p.49]{Rain-NY-1960}):\\
	The function converges to a finite value given by the ratio of gamma functions
	\begin{equation}\label{c>a+b}
		F(a,b;c;1^-)=\frac{\Gamma(c)\Gamma(c-a-b)}{\Gamma(c-a)\Gamma(c-b)}.
	\end{equation}
	\item $c=a+b$ (c.f. \cite[Sec. 1.48]{AVV-NY-1997}):\\
	In this case, Ramanujan's asymptotic formula describes the behavior as $x\to1^-$,
	\begin{equation}\label{c=a+b}
	B(a,b)F(a,b;c;x)+\ln(1-x)=R(a,b)+O[(1-x)\ln(1-x)].
	\end{equation}
	\item $c<a+b$ (c.f. \cite[15.3.3]{AS-Wash-1964}):\\
	The function exhibits a singular behavior near $x=1$, with the leading term determined b
	\begin{equation}\label{c<a+b}
		F(a,b;c;x)=\frac{\Gamma(c)\Gamma(a+b-c)}{\Gamma(a)\Gamma(b)}(1-x)^{c-a-b}(1+o(1)).
	\end{equation}
	A useful symmetry relation for this case is
	\begin{equation}\label{sym-relation}
		F(a,b;c;x)=(1-x)^{c-a-b}F(c-a,c-b;c;x).
	\end{equation}
\end{itemize}

As the special case of zero-balanced hypergeometric function, the Legrendre's complete elliptic integral of the first kind (with the modulus $r\in(0,1)$)  can be explicitly expressed in terms of $F(a,b;c;x)$  [15, 16]. This integral, which plays a central role in geometric function theory, elasticity, and number theory, is defined as:
\begin{equation}
	\mathcal{K}(r)=\int_{0}^{\pi/2}(1-r^2\sin \theta)^{-1/2}d\theta=\frac{\pi}{2}F\left(\frac{1}{2},\frac{1}{2};1;r^2\right).
\end{equation}
Recent years have seen growing interest in the absolute monotonicity of functions derived from Gaussian hypergeometric functions (and their special cases like elliptic integrals), driven by its applications in proving sharp inequalities, characterizing convex/concave functions, and analyzing asymptotic behavior (see \cite{CWZ-RCSM-2023,CZ-RIM-2022,HQM-JMAA-2019,TY-RCSM-2023,TY-RIM-2022,Yang-MIA-21-2018,ZWC-RCSM-2021}). A function $f(x)$ is said to be absolutely monotonic \cite{Wid-Prin-1941} (``AM'' for short) on an interval $I$ if all its derivatives $f^{(k)}(x)$ exist and satisfy $f^{(k)}(x) \geq 0$ for all $x \in I$ and $k\in\mathbb{N}_{0}$.

Early work in this area focused on functions related to $\mathcal{K}(r)$. In 2022, Yang and Tian \cite{YT-AMS-2022} studied the absolute monotonicity of $f_p(x)=(1-x)^p\mathcal{K}(\sqrt{x})$ and $g_p(x)=(1-x)^p\exp(\mathcal{K}(\sqrt{x}))$, proving:
\begin{itemize}[leftmargin=2em]
	\item $-f'_p$ is absolutely monotonic on $(0,1)$ if and only if $\frac{1}{4}\leq p\leq1$;
	\item $-\left(\ln f_p\right)'$ is absolutely monotonic on $(0,1)$ if and only if $p\geq\frac{1}{4}$; 
	\item In a subsequent paper \cite{YT-JMI-2021}, they further showed that $g'_p$ is absolutely monotonic on $(0,1)$ if and only if $p\leq \pi/8$, whilie $-g'_p$ is 
	absolutely monotonic on $(0,1)$ if and only if $\frac{1}{2}\leq p\leq \frac{\pi+4+\sqrt{16-\pi}}{8}$.
\end{itemize} 
These results have since been extended to more general zero-balanced hypergeometric functions. For instance: 
\begin{itemize}[leftmargin=2.8em]
\item Wu and Zhao \cite{WZ-IJPAM-2025} studied the power series expansion of $(1-x)^{p} F(a, b ; a+b ; x)$, deriving a sufficient condition for $-\left[(1-x)^{p} F(a, b ; a+b ; x)\right]'$ to be absolutely monotonic on $(0, 1)$;
\item They also provided necessary and sufficient conditions for the absolute monotonicity of $\ln\left[(1-x)^{p} F(a, b ; c ; x)\right]$ (and its higher derivatives) in \cite{WZ-BIMS-2024}, though the resulting parameter inequalities are computationally cumbersome;
\item Very recently, Sun, Wang and Huang \cite{SWH-JIA-2025} proved that  $(1-x)^{p} \exp(F(a, b ; a+b ; x))$ is absolutely monotonic on $(0, 1)$ if and only if $p \leq \frac{ab}{a+b}$, and the function $-\left[(1-x)^{p} \exp(F(a, b ; a+b ; x))\right]'$ is absolutely monotonic if and only if $\frac{\Gamma(a+b)}{\Gamma(a)\Gamma(b)}\leq p\leq \frac{ab}{a+b}+\frac{1}{2}\left[1+\sqrt{\frac{a^2+b^2-4a^2b^2+2ab+a+b}{(a+b)(a+b+1)}}\right]$. However, this result is restricted to $a,b\in (0, 1]$, limiting its applicability.
\end{itemize} 

Existing studies on absolute monotonicity of hypergeometric-derived functions suffer from two key limitations:
\begin{itemize}[leftmargin=2em]
	\item[1.] {\it Cumbersome Proof Techniques:} Most works \cite{SWH-JIA-2025,WZ-BIMS-2024,WZ-IJPAM-2025} rely on recursive coefficient formulas and mathematical induction to analyze the sign of power series coefficients. This approach is not only technically complicated but also forces restrictive parameter ranges to satisfy induction hypotheses, narrowing the validity of results.
	\item[2.] {\it Limited Generality:} Results are often confined to specific subclasses (e.g., zero-balanced functions with $a,b\in(0,1]$) or lack unified conditions for absolute monotonicity of higher derivatives.
\end{itemize}

To address these gaps, we focus on two general families of functions defined on $(0, 1)$ with $a,b,c\in\mathbb{R}_+$ by 
\begin{equation*}
	\mathcal{F}_p(x)=(1-x)^pF(a,b;c;x)\quad \text{and}\quad \mathcal{G}_p(x)=(1-x)^p \exp(F(a,b;c;x)).
\end{equation*}
Our core contribution is a novel, constructive method to prove absolute monotonicity, leveraging Jurkat’s criterion \cite{Jurkat-PAMS-1954} for the ratio of power series. The method proceeds in four key steps:
\vspace{-0.2cm}
\begin{itemize}[leftmargin=2.5em]
	\item[(i)] Rewrite $\mathcal{F}_p(x),$ $\mathcal{G}_p(x)$ or $\ln\mathcal{F}_p(x)$ as a ratio of  two power series: $\dfrac{\sum_{n=0}^{\infty}u_nx^n}{\sum_{n=0}^{\infty}v_nx^n}$;
	\vspace{-0.18cm}
	\item[(ii)] Prove the sequence $\left\{v_{n+1}/v_n\right\}_{n\geq0}$ is  increasing;
	\item [(iii)] Prove the sequence $\left\{u_n/v_n\right\}_{n\geq0}$  is increasing or decreasing (the most challenging step, which we solve by constructing an auxiliary absolutely monotonic function);
	\item[(iv)] Apply Jurkat’s criterion to conclude the absolute monotonicity of  $\pm\mathcal{F}'_p,$ $\pm\mathcal{G}'_p$ and $\pm\left(\ln\mathcal{F}_p\right)'$.
\end{itemize}
This approach avoids the complexity of induction and recursive coefficient analysis, enabling us to derive necessary and sufficient conditions for absolute monotonicity (and higher derivatives) of $\mathcal{F}_{p}(x)$ and $\mathcal{G}_{p}(x)$ over broad parameter ranges of $a, b, c$. As applications, we also establish sharp inequalities for Gaussian hypergeometric functions and extend the parameter validity of existing inequalities.

The rest of this paper is as follows.  In Section 2, we present foundational tools (Jurkat’s criterion and a new test function for sequence monotonicity) and technical lemmas, including analyses of parameter conditions and several absolute monotonicity. In Section 3, we will derive the main results: necessary and sufficient conditions for the absolute monotonicity of $-\mathcal{F}_{p}'$, $\pm \mathcal{F}_{p}''$, $\pm\mathcal{G}_{p}'$, and $\pm(\ln\mathcal{F}_{p})^{(k)}$ ($k \geq 1$). Applying the main results, Section 4 closes the paper by establishing new inequalities for Gaussian hypergeometric functions and improving the parameter range of existing inequalities.

\section{Preliminary}

In this section, we present two key tools and several technical lemmas to establish the absolute monotonicity of functions involving Gaussian hypergeometric functions. 

\subsection{Tools}
The core tool for analyzing power series ratios is Jurkat’s criterion \cite{Jurkat-PAMS-1954}, which we restate with refined clarity:
\begin{proposition}{\bf (Jurkat's criterion)}\label{prop-Jurkat}
	Let $p(x)=\sum_{n=0}^{\infty} p_nx^n$ and $q(x)=\sum_{n=0}^{\infty}q_nx^n$ be power series convergent on $(0,r)$ with $q_n>0$ for all $n\geq0$. Suppose that the sequence $\{q_{n+1}/q_n\}_{n\geq0}$ is increasing.  If the sequence $\{p_n/q_n\}_{n\geq0}$ is increasing (resp. decreasing), then the function $(p/q)'$ (resp. $-(p/q)'$) is absolutely monotonic on $(0,r)$.
\end{proposition}
To verify the sequence monotonicity required by Proposition \ref{prop-Jurkat}, we introduce a {\bf test function} to determine the monotonicity of the sequence.
Let \(p \geq 0\) and define two convergent power series on $(0, 1)$:
 \begin{equation}\label{phi-Wn}
 	\phi(x)=\sum_{n=0}^{\infty}V_nx^n\quad \text{and}\quad (1-x)^{-p}=\sum_{n=0}^\infty W_nx^n,
 \end{equation}
where $W_n:=W_n(p)=\frac{(p)_n}{n!}=\frac{\Gamma(n+p)}{\Gamma(p)\Gamma(n+1)}$ (here, $(p)_n$ denotes the Pochhammer symbol: $(p)_0=1$ and $(p)_n=p(p+1)\cdots(p+n-1)$ for $n\geq1$).
 
\begin{proposition}\label{prop-method}
	Define the test function of $\phi(x)$ as: $	\mathcal{T}_p\phi(x)=(1-x)\phi'(x)-p\phi(x)$. Then:
	\begin{itemize}[leftmargin=2em]
		\item The sequence $\{ V_n/W_n \}_{n \geq 0}$ is increasing (resp. decreasing) if and only if $\mathcal{T}_p \phi(x)$\\ (resp. $-\mathcal{T}_p \phi(x)$) is absolutely monotonic on $(0, 1)$.
		\item  In particular, $\{n V_n\}_{n\geq0}$ is increasing (resp. decreasing) if and only if the function $\mathcal{T}_0\phi(x)$ (resp. $-\mathcal{T}_0\phi(x)$) is absolutely monotonic on $(0,1)$. 
	\end{itemize}
\end{proposition} 
\begin{proof}
Expand $\mathcal{T}_p\phi(x)$ via power series:
\begin{equation*}
\mathcal{T}_p\phi(x)=(1-x)\sum_{n=1}^{\infty}nV_nx^{n-1}-p\sum_{n=0}^{\infty} V_nx^n=\sum_{n=0}^{\infty}\Big[(n+1)V_{n+1}-(n+p)V_n\Big] x^n.
\end{equation*}
Substitute $W_{n+1}=\frac{(p)_{n+1}}{(n+1)!}=\frac{n+p}{n+1}W_n$ (from Pochhammer recurrence). Rearranging gives
\begin{equation}\label{V_n/W_n}
		\mathcal{T}_p\phi(x)=\sum_{n=0}^\infty(n+1)W_{n+1}\left(\frac{V_{n+1}}{W_{n+1}}-\frac{V_n}{W_n}\right)x^n.
\end{equation}
Since $W_{n+1}>0$ for all $n\geq0$ and $p\geq0$, the sign of $\mathcal{T}_p\phi(x)$'s  coefficients equals the sign of $\frac{V_{n+1}}{W_{n+1}}-\frac{V_n}{W_n}$. Specifically,  $\mathcal{T}_0\phi(x)$'s  coefficients is simplified to $(n+1)V_{n+1}-nV_n$. Thus, the conclusion is valid.
\end{proof}

\subsection{Lemmas} In this subsection, we present four necessary lemmas to prove the main results.
\begin{lemma}\label{lemma-W_{n+1}/W_n}
	The sequence $\{W_{n+1}/W_n\}_{n\geq0}$ is increasing if and only if $p\leq1$.
\end{lemma}

\begin{proof}
	Compute the difference of consecutive terms:
	\begin{equation*}
		\begin{split}
			\frac{W_{n+2}}{W_{n+1}}-\frac{W_{n+1}}{W_n}&=\frac{(p)_{n+2}/(n+2)!}{(p)_{n+1}/(n+1)!}-\frac{(p)_{n+1}/(n+1)!}{(p)_n/n!}\\
			&=\frac{p+n+1}{n+2}-\frac{p+n}{n+1}=\frac{1-p}{(n+1)(n+2)}.
		\end{split}
	\end{equation*}
	For all $n \geq 0$, $(n+1)(n+2) > 0$. Thus:
	\begin{itemize}
		\item If $p \leq 1$, the difference is non-negative, so $\{ W_{n+1}/W_n\}$ is increasing.
		\item If $p > 1$, the difference is negative, so the sequence is decreasing.
	\end{itemize}
	The proof is completed.
\end{proof}

The following lemma can be easily derived from the properties of quadratic functions.
\begin{lemma}\label{lemma-tau}
	Let $a,b,c\in\mathbb{R}_+$ with $(c-a)(c-b)\geq0$, and define
	\begin{equation*}
		\tau(p):=\tau(p;a,b,c)=p^2-\left(1+\frac{2ab}{c}\right)p+\frac{ab(a+1)(b+1)}{c^2(c+1)}.
	\end{equation*}
	Then $\tau(p)\leq0$ if and only if $p_*\leq p\leq p^*$, where  $p_*$ (the smaller root) and $p^*$ (the larger root) of $\tau(p)$ are explicitly computed as:
	\begin{equation*}
		p_*=\frac{1}{2}+\frac{ab}{c}-\sqrt{\frac{1}{4}+\frac{ab(c-a)(c-b)}{c^2(c+1)}},\quad p^*=\frac{1}{2}+\frac{ab}{c}+\sqrt{\frac{1}{4}+\frac{ab(c-a)(c-b)}{c^2(c+1)}}.
	\end{equation*}Moreover, $0<p_*<\frac{ab}{c}<1+\frac{ab}{c}<p^*$. In particular, if $a+b\leq c$, then 
	\begin{equation}\label{lemma2-1}
		\frac{1+a+b+ab}{c+2}<1+\frac{ab}{c}.
	\end{equation}
\end{lemma}
\begin{proof}
	For $(c-a)(c-b)\geq0$, the quadratic $\tau(p)$ has discriminant:
	\begin{equation*}
	\Delta=\left(1+\frac{2ab}{c}\right)^2-4\frac{ab(a+1)(b+1)}{c^2(c+1)}=1+\frac{4ab(c-a)(c-b)}{c^2(c+1)}>0,
	\end{equation*}
	so it has two distinct real roots $p_*$ (smaller) and $p^*$ (larger). 
	Moreover,
	\begin{equation*}
	0<p_*<\frac{1}{2}+\frac{ab}{c}-\frac{1}{2}=\frac{ab}{c}<1+\frac{ab}{c}=\frac{1}{2}+\frac{ab}{c}+\frac{1}{2}<p^*.
	\end{equation*}
	For $a+b\leq c$, verify the inequality:
	\begin{equation*}
		\frac{1+a+b+ab}{c+2}-\left(1+\frac{ab}{c}\right)=-\frac{2ab+c+c(c-a-b)}{c(c+2)}<0,
	\end{equation*}
since $c(c-a-b)\geq0$ and $2ab>0$.
\end{proof}

\begin{lemma}\label{lemma-F}
	Let $a,b,c\in\mathbb{R}_+$ with $a+b\leq c$ and define
	\begin{align*}
		F(x)&=p(1-p)F(a,b;c;x)+\frac{2abp}{c}(1-x)F(a+1,b+1;c+1;x)\\
		&\quad -\frac{ab(a+1)(b+1)}{c(c+1)}(1-x)^2F(a+2,b+2;c+2;x),
	\end{align*}
    where $p_*$ and $p^*$ are given in Lemma \ref{lemma-tau}.
    Then the following statements hold:
	\begin{enumerate}[leftmargin=1.6em,label=(\roman*)]
		\item $\dfrac{F(x)}{1-x}$ is AM on $(0,1)$ if and only if $p_*\leq p\leq 1$;
		\item If $p\geq p^*$, then $-F(x)$ is AM on $(0,1)$.
	\end{enumerate}
\end{lemma}

\begin{proof}
	First, we expand $F(x)$ into a power series to analyze the sign of its coefficients. Using the power series representation of the Gaussian hypergeometric function, we rewrite $F(x)$ as 
	\begin{align}\nonumber
		F(x)&=p(1-p)\sum_{n=0}^\infty\frac{(a)_n(b)_n}{(c)_nn!}x^n+\frac{2abp}{c}(1-x)\sum_{n=0}^\infty\frac{(a+1)_n(b+1)_n}{(c+1)_nn!}x^n+\\ \nonumber
		&\qquad-\frac{ab(a+1)(b+1)}{c(c+1)}(1-x)^2\sum_{n=0}^\infty\frac{(a+2)_n(b+2)_n}{(c+2)_nn!}x^n\\ \label{F(x)}
		&=\sum_{n=0}^\infty \frac{(a)_n(b)_n\kappa_n}{(c)_nn!}x^n,
	\end{align}
	where the coefficient $\kappa_n$ (dependent on $p,a,b,c$) is
	\begin{align*}
		\kappa_n&=-p^2+\left[1+\frac{2ab+2(a+b-c)n}{c+n}\right]p\\
		&\quad-\frac{(a+b-c)n\left[\begin{array}{l}
				1+a+b+2ab\\+c+n+(a+b-c)n
			\end{array}\right]+ab(a+1)(b+1)}{(n+c)(n+c+1)}.
	\end{align*}
	A key recurrence relation for $\kappa_n$ reveals its monotonicity:
	\begin{equation*}
		\kappa_{n+1}-\kappa_n=\frac{2(c-a)(c-b)[1+a+b+ab-(c+2)p-(p+c-a-b)n]}{(c+ n) (1+c+n) (2+c+n)}.
	\end{equation*}
	Since $a+b\leq c$ and $(c-a)(c-b)\geq0$ (from Lemma \ref{lemma-tau}'s condition),  $\kappa_n$ is either decreasing (if $1+a+b+ab-(c+2)p\leq0$), or unimodal (first increasing, then decreasing; if $1+a+b+ab-(c+2)p>0$) for all $n\geq0$. 
	
\noindent(i) Absolute monotonicity of $\frac{F(x)}{1-x}$.
	\begin{description}[leftmargin=0.5em]
		\item[Necessity]  If $\frac{F(x)}{1-x}$ is absolutely monotonic, its power series coefficients are non-negative for all $n\geq0$. At $x=0$, $F(0)=\kappa_0=-\tau(p)\geq0$, so Lemma \ref{lemma-tau} implies $p_*\leq p\leq p^*$. Additionally, for $p>1$,
		we analyze the limit of $F(x)$ as $x\to1^-$:
		\begin{itemize}[leftmargin=1.6em]
			\item[$\circ$] If $a+b=c$, $F(1^-)=-\infty$ (singular behavior of zero-balanced hypergeometric functions, via \eqref{c=a+b});
			\item[$\circ$] If $a+b<c$, $F(1^-)=-p(p-1)\frac{\Gamma(c)\Gamma(c-a-b)}{\Gamma(c-a)\Gamma(c-b)}<0$ (finite but negative, via \eqref{c>a+b}).
		\end{itemize}
		In both cases, $\lim_{x\to1^-}\frac{F(x)}{1-x}=-\infty$, which contradicts the non-negativity of all derivatives of an absolutely monotonic function. Thus $p\leq1,$ and thereby $p_*\leq p\leq 1$.
		\item[Sufficiency] For $p_*\leq p\leq 1$, we split into two cases based on the value of $a+b-c+1$ (a threshold derived from the limit $\kappa_\infty=(p+c-a-b)(a+b-c+1-p)$):
		\begin{itemize}[leftmargin=1.5em]
		\item[$1.$] If $p_*\leq p\leq a+b-c+1$, then $\kappa_n\geq\min\{\kappa_0,\kappa_\infty\}\geq0$.
			Here, $\kappa_0=-\tau(p)\geq0$ (by $p_*\leq p\leq p^*$) and $\kappa_\infty\geq0$ (by $p\leq a+b-c+1$), so all $\kappa_n\geq0$. Since $\frac{F(x)}{1-x}$ is the product of two absolutely monotonic functions (the power series \eqref{F(x)} and $\frac{1}{1-x} = \sum_{n=0}^{\infty}x^n$), it is absolutely monotonic.
		\item[$2.$] If $a+b-c+1<p\leq 1$, then $a+b<c$ (so $\kappa_\infty<0$), and $\kappa_n$ is unimodal (first increasing, then decreasing).
		Let $\frac{F(x)}{1-x}=\sum_{n=0}^{\infty}\kappa_n^*x^n$, then $F(x)=(1-x)\sum_{n=0}^\infty \kappa^*_nx^n=\sum_{n=0}^\infty(\kappa^*_n-\kappa^*_{n-1})x^n$ (with $\kappa^*_{-1}=0$).  By \eqref{F(x)},$\kappa^*_n-\kappa^*_{n-1}=\frac{(a)_n(b)_n\kappa_n}{(c)_nn!}$, so $\{\kappa^*_n\}$ is also unimodal (first increasing, then decreasing). This gives  $\kappa^*_n\geq\min\{\kappa^*_0,\kappa^*_\infty\}=\min\{-\tau(p),F(1^-))\}$, both of which are non-negative (by $p_*\leq p\leq p^*$ and $a+b<c$). Thus all $\kappa^*_n\geq0$, namely, $\frac{F(x)}{1-x}$ is absolutely monotonic.
			\end{itemize}
	\end{description}
	
	\medskip
\noindent(ii) Absolute monotonicity of $-F(x)$.
	
	 For $p\geq p^*$, Lemma \ref{lemma-tau} implies $\tau(p)\geq0$. By \eqref{lemma2-1}, we have  $(c+2)p>1+a+b+ab$, and thereby $\kappa_n$ is decreasing for $n\geq0$. Hence, $\kappa_n\leq \kappa_0=-\tau(p)\leq0$. Combining this with \eqref{F(x)}, we derive that $-F(x)$ is absolutely monotonic on $(0,1)$.
\end{proof}

\begin{lemma}\label{lemma-H}
	Let $a,b,c\in\mathbb{R}_+$ with $c\geq a+b-1$ and $(1+a+b-ab)c\geq2ab$, and define
	\begin{equation*}
		H(x)=\frac{(1-x)F(a+1,b+1;c+1;x)}{F(a,b;c;x)}.
	\end{equation*}
	Then we have:
	\begin{itemize}[leftmargin=2em]
		\item If $(c-a)(c-b)\leq0$, then $H'(x)$ is AM on $(0,1)$;
		\item If $(c-a)(c-b)\geq0$,  then $-H'(x)$ is AM on $(0,1)$.
	\end{itemize}
\end{lemma}

\begin{proof}
 Write $H(x)$ in terms of power series as
	\begin{equation*}
		H(x)=\frac{\sum_{n=0}^\infty \mu_nx^n}{\sum_{n=0}^\infty \lambda_nx^n},
	\end{equation*}
	where 
	\begin{equation}\label{mu-lambda}
		\mu_n=\frac{c[ab+(a+b-c)n]}{ab(n+c)}\lambda_n	 \quad\text{and}\quad	\lambda_n=\frac{(a)_n(b)_n}{(c)_nn!}.
	\end{equation}
	\begin{itemize}[leftmargin=1.2em]
		\item Monotonicity of $\{\lambda_{n+1}/\lambda_n\}$.\\
		For $n\geq0$,  we have $\lambda_{n+1}/\lambda_n=[(a+n)(b+n)]/[(c+n)(1+n)]$ and thereby
		\begin{align*}
		&\frac{\lambda_{n+2}}{\lambda_{n+1}}-\frac{\lambda_{n+1}}{\lambda_n}=\frac{\left[\begin{array}{l}
				(1+a+b-ab)c-2ab+ (1-a-b-2ab+3c)n\\+(1-a-b+c)n^2
		\end{array}\right]}{(1+n)(2+n)(c+n)(1+c+n)}\geq0,
		\end{align*}by $c \geq a+b-1$ and $(1+a+b-ab)c \geq 2ab$.
       \item Monotonicity of $\{\mu_{n}/\lambda_n\}$.\\
		By \eqref{mu-lambda}, we simplify the ratio $\mu_n/\lambda_n$ to
		\begin{equation*}
			\frac{\mu_n}{\lambda_n}=\frac{c[ab+(a+b-c)n]}{ab(n+c)}.
		\end{equation*}
		The difference of consecutive ratios is
		\begin{equation*}
			\frac{\mu_{n+1}}{\lambda_{n+1}}-\frac{\mu_{n}}{\lambda_n}=-\frac{c(c-a)(c-b)}{a b (c + n) (1 + c + n)}. 
		\end{equation*}
		If $(c-a)(c-b)\leq0$, the difference is non-negative, so  $\mu_n/\lambda_n$ is increasing. By Proposition \ref{prop-Jurkat},  so $H'(x)$ is absolutely monotonic on $(0,1)$. If $(c-a)(c-b)\geq0$, the difference is non-negative, so  $\mu_n/\lambda_n$ is decreasing. By Proposition \ref{prop-Jurkat},  so $-H'(x)$ is absolutely monotonic on $(0,1)$
	\end{itemize}
	This completes the proof.
\end{proof}

\section{Several absolutely monotonic functions}

This section systematically derives the conditions for determining the absolute monotonicity of two core function families $\mathcal{F}_p(x)$ and $\mathcal{G}_p(x)$ as well as the logarithmic transform $\ln\mathcal{F}_p(x)$, based on the foundational tools (Jurkat’s criterion, test function for sequence monotonicity) and lemmas \ref{lemma-W_{n+1}/W_n}--\ref{lemma-H} in Section 2. 

\subsection{Absolute Monotonicity of $\mathcal{F}_p$}	

The absolute monotonicity of $\mathcal{F}_p(x)$ and its first two derivatives is analyzed in the cases of $a+b\leq c$ and $a+b>c$, by using power series ratio properties and Lemma \ref{lemma-F}.

\begin{theorem}\label{theorem-Fp}{\bf (Case $c\geq a+b$)}\\
	Let $c\geq a+b$ and $\mathcal{F}_p(x)=(1-x)^pF(a,b;c;x)$ be defined on $(0,1)$. 
	Let $p_*$ and $p^*$ denote the roots of $\tau(p)$ in Lemma \ref{lemma-tau} (i.e. $p_*=\frac{1}{2}+\frac{ab}{c}-\sqrt{\frac{1}{4}+\frac{ab(c-a)(c-b)}{c^2(c+1)}}$and $p^*=\frac{1}{2}+\frac{ab}{c}+\sqrt{\frac{1}{4}+\frac{ab(c-a)(c-b)}{c^2(c+1)}}$). Then:
	\begin{enumerate}[leftmargin=2.5em,itemindent=0em, label=(\roman*)]
		\item $-\mathcal{F}'_p$ is AM on $(0,1)$ if and only if $\frac{ab}{c}\leq p\leq1$;
		\item $-\mathcal{F}''_p$ is AM on $(0,1)$ if and only if $p_*\leq p\leq1$;
		\item If $p^*\leq p\leq 2$, then $\mathcal{F}''_p$ is AM on $(0,1)$.
	\end{enumerate}
\end{theorem}

\begin{proof}
	(i) \textbf{\textsl{Necessity.}}  If $-\mathcal{F}'_p$ is absolutely monotonic on $(0,1)$, 
	all its derivatives on $(0, 1)$ are non-negative. In particular, the right-hand limit of its first derivative at $x=0^+$ must satisfy non-positivity:
	\begin{equation*}
		\mathcal{F}'_p(0^+)=\lim_{x\to0^+}\frac{d}{dx}\Big[(1-x)^{p}F(a,b;c;x)\Big]=\frac{ab}{c}-p\leq0,
	\end{equation*}which implies $p\geq\frac{ab}{c}$.
	To further prove $p\leq1$, rewrite $\mathcal{F}_p(x)=(1-x)\mathcal{F}_{p-1}(x)$, where $\mathcal{F}_{p-1}(x)=(1-x)^{p-1}F(a,b;c;x)=\sum_{n=0}^\infty A^*_nx^n.$
	Expanding this expression gives $\mathcal{F}_p(x)=1+\sum_{n=1}^\infty(A^*_n-A^*_{n-1})x^n$. From the absolute monotonicity of $-\mathcal{F}'_p$, it follows that $A^*_n-A^*_{n-1}\leq0$ (i.e $\{A^*_n\}_{n\geq0}$ is decreasing.)
	If $p>1$, then $p-1>0$. Combining the finiteness of $F(a,b;c;1^-)$ when $c\geq a+b$ (see \eqref{c>a+b}) or the asymptotic behavior in the zero-balanced case (see \eqref{c=a+b}), we obtain $\mathcal{F}_{p-1}(1^-)=0$ (i.e. $\sum_{n=0}^\infty A^*_n=0<\infty$). Thus, $\{A^*_n\}_{n\geq0}$ is decreasing and converges to $0$, which implies $A^*_n\geq0$; however, $\mathcal{F}_{p-1}(0)=1$, meaning $A^*_0=1,$ which is a contradiction. Therefore, $p\leq1$, and collectively, $\frac{ab}{c}\leq p\leq 1$.
	
	\textbf{\textsl{Sufficiency.}} 
	In accordance with Jurkat’s criterion (Proposition 1), \(\mathcal{F}_p(x)\) needs to be expressed as the ratio of two power series with positive coefficients to determine absolute monotonicity using sequence monotonicity. 
	
	Note that $(1-x)^{-p}=\sum_{n=0}^\infty W_n x^n$ (where $W_n=\frac{(p)_n}{n!}$ is defined in \eqref{phi-Wn}), so:
	\begin{equation*}
		\mathcal{F}_p(x)=\frac{F(a,b;c;x)}{(1-x)^{-p}}=\frac{f(x)}{(1-x)^{-p}}=\frac{\sum_{n=0}^{\infty} A_nx^n}{\sum_{n=0}^\infty W_nx^n},
	\end{equation*}	where $f(x)=\sum_{n=0}^\infty A_n x^n$ ($A_n=\frac{(a)_n(b)_n}{(c)_nn!}$). Construct the test function $-\mathcal{T}_pf(x)$ for $f(x)$ as defined in Proposition \ref{prop-method}; expanding this function yields:
	\begin{equation*}
		-\mathcal{T}_pf(x)=\sum_{n=0}^{\infty}\left[p-\frac{ab}{c}+\frac{(c-a)(c-b)n}{c(c+n)}\right]A_nx^n.
	\end{equation*}
	Since $p\geq\frac{ab}{c}$ and $c\geq a+b$ implies $(c-a)(c-b)\geq0$, all coefficients of this series are non-negative, so $-\mathcal{T}_pf(x)$ is absolutely monotonic.  By Proposition \ref{prop-method}, $\{A_n/W_n\}_{n\geq0}$ is decreasing; by Lemma \ref{lemma-W_{n+1}/W_n}, $\{W_{n+1}/W_n\}_{n\geq0}$ is increasing when $p\leq1$. Combining these results with Jurkat's criterion (Proposition \ref{prop-Jurkat}), $-\mathcal{F}'_p$ is absolutely monotonic, which completes the proof of sufficiency.
	
	\medskip
	
	(ii)	If $-\mathcal{F}''_p$ is absolutely monotonic, then $\mathcal{F}''_p(0^+)=\tau(p)\leq0$ (where $\tau(p)$ is defined in Lemma \ref{lemma-tau}). By Lemma \ref{lemma-tau}, this implies $p_*\leq p\leq p^*$. Combining this with necessity of $p\leq1$ (similarly, if $p>1$, $\mathcal{F}''_p(1^-)\to-\infty$, which contradicts absolutely monotonicity), we have $p\in[p_*,1].$
	
	For sufficiency, express $\mathcal{F}'_p(x)$ as a ratio of two power series:
	\begin{equation}\label{F'_p}
		\mathcal{F}'_p(x)=\frac{\tilde{f}(x)}{(1-x)^{-p}}=\frac{\sum_{n=0}^\infty \tilde{A}_n x^n}{\sum_{n=0}^\infty W_nx^n},
	\end{equation}
	where $\tilde{f}(x)=\frac{ab}{c}F(a+1,b+1;c+1;x)-\frac{pF(a,b;c;x)}{1-x}.$
	At this point, the test function	$-\mathcal{T}_p\tilde{f}(x)=\frac{F(x)}{1-x}$ (where $F(x)$ is defined in Lemma \ref{lemma-F}) is  absolutely monotonic by Lemma \ref{lemma-F}(i) when $p_*\leq p\leq1]$. Combining this with Proposition \ref{prop-method} and Jurkat's criterion, it follows that  $-\mathcal{F}''_p$ is absolutely monotonic. 
	
	\medskip
	
	(iii) Direct differentiation of $\mathcal{F}_p(x)$ gives
	\begin{align*}
	\mathcal{F}''_p(x)&=\left[-p(1-x)^{p-1}F(a,b;c;x)+(1-x)^p\frac{ab}{c}F(a+1,b+1;c+1;x)\right]\\
	&=-(1-x)^{p-2}F(x).
	\end{align*}When $p^*\leq p\leq 2$, Lemma \ref{lemma-F}(ii) implies $-F(x)$ is absolute monotonic; additionally, when $p\leq2$, $(1-x)^{p-2}=\sum_{n=0}^{\infty}\frac{(2-p)_n}{n!}x^n$ is absolute monotonic (all coefficients are non-negative). The product of two absolutely monotonic functions is also absolutely monotonic, so $\mathcal{F}''_p$ is absolutely monotonic.
\end{proof}
\begin{remark}
To clarify the necessity of the parameter range for the absolute monotonicity of $\mathcal{F}_p''(x)$ in Theorem \ref{theorem-Fp}(iii), we begin with the explicit expression of the second derivative of $\mathcal{F}_p(x)$, derived via direct differentiation and substitution of the function $F(x)$ defined in Lemma \ref{lemma-F}: $\mathcal{F}''_p(x)=-(1-x)^{p-2}F(x)$.
For $\mathcal{F}_p''(x)$ to be absolutely monotonic, $\mathcal{F}''_p(0)=-F(0)=\tau(p)\geq0$, which holds if and only if $p\leq p_*$ or $p\geq p^*$ (with $p_*,p^*$ being the roots of $\tau(p)$). On the other hand,  or $p>2$,  $\mathcal{F}_p(1^-)=-\infty$ (if $0<p<1$ since $0<F(1^-)<\infty$ ) and $\mathcal{F}_p(1^-)=0$ (if $p>2$ since $F(1^-)<\infty$), each of which the Maclaurin coefficients of $\mathcal{F}''_p(x)$ fail to satisfy non-negative for all $n$. Thus, $p$ can only satisfy $1\leq p\leq2$, and collectively, $1\leq p\leq p^*$ or $p^*\leq p\leq 2$. 
However, we cannot currently prove that $\mathcal{F}_p''(x)$ is absolutely monotonic when $1 \leq p \leq p_*$: for this range, the coefficients of $\mathcal{F}_p''(x)$ (derived from the power series of $F(x)$) do not satisfy the non-negativity required for absolute monotonicity (unlike the $p^* \leq p \leq 2$ case, where Lemma \ref{lemma-F} guarantees $-F(x)$ is absolutely monotonic, making all the coefficients of $\mathcal{F}_p''(x) = -(1-x)^{p-2} F(x)$ be non-negative). This leaves $p^*\leq p\leq2$ as the only rigorously verified range for the absolute monotonicity of $\mathcal{F}_p''(x)$ in Theorem \ref{theorem-Fp}(iii).
\end{remark}

\begin{theorem}\label{theorem-Fp-2}{\bf (Case $c<a+b$)}\\
Let $\max\{a,b\}<c<a+b$ and define $\mathcal{F}_p$ as in Theorem \ref{theorem-Fp}. Let $p_*$ and $p^*$ be the same as those in Lemma \ref{lemma-tau}. Then:
\begin{enumerate}[leftmargin=2.5em,itemindent=0em, label=(\roman*)]
	\item $-\mathcal{F}'_p$ is AM on $(0,1)$ if and only if $\frac{ab}{c}\leq p\leq a+b+1-c$;
	\item $-\mathcal{F}''_p$ is AM on $(0,1)$ if and only if $p_*\leq p\leq a+b+1-c$;
	\item If $p^*\leq p\leq a+b+2-c$, then $\mathcal{F}''_p$ is AM on $(0,1)$.
\end{enumerate}
\end{theorem}
\proof
Use the symmetry relation for Gaussian hypergeometric functions \eqref{sym-relation},
we rewrite $\mathcal{F}_p(x)$ as
\begin{equation*}
	\mathcal{F}_p(x)=(1-x)^{\bar{p}}F(\bar{a},\bar{b};c;x),
\end{equation*} 
where $\bar{a}=c-a$, $\bar{b}=c-b$ and $\bar{p}=p+c-a-b$.
Since $c<a+b$ and $c>\max\{a,b\}$, we have  $\bar{a}+\bar{b}=2c-(c+b)<c$ and $\bar{a},\bar{b}>0$. Thus, $\bar{a},\bar{b},c$ satisfy the $c\geq a+b$ case in Theorem \ref{theorem-Fp}.  Translating the conclusions of Theorem \ref{theorem-Fp} back to original parameters:
\begin{enumerate}[leftmargin=2.5em,label=(\roman*)]
	\item $\bar{a}\bar{b}/c\leq \bar{p}\leq1\Longrightarrow \frac{ab}{c}\leq p\leq a+b+1-c$;
	\item $\bar{p}_*\leq \bar{p}\leq1\Longrightarrow p_*\leq p\leq a+b+1-c$, where $\bar{p}_*=p^*+c-a-b$ is the smaller root of quadratic function $\bar{\tau}(p)=\tau(p;\bar{a},\bar{b},c)$. 
	\item $\bar{p}^* \leq \bar{p} \leq 2 \Longrightarrow p^* \leq p \leq a+b+2 - c$ where $\bar{p}^* = p^* + c - a - b$ is the larger root of quadratic function $\bar{\tau}(p)$. \qed
\end{enumerate}

\begin{corollary}\label{corollary1}{\bf (Zero-Balanced Case $c=a+b$)}\\
For $c=a+b$, the function $\mathcal{F}_p(x)=(1-x)^pF(a,b;a+b;x)$ satisfies:
\begin{enumerate}[leftmargin=2em,itemindent=0em, label=(\roman*)]
	\item $-\mathcal{F}'_p$ is AM on $(0,1)$ if and only if $\frac{ab}{a+b}\leq p\leq1$.
	\item $-\mathcal{F}''_p$ is AM on $(0,1)$ if and only if 
	$\frac{1}{2}+\frac{ab}{a+b}-\sqrt{\frac{1}{4}+\frac{a^2b^2}{(a+b)^2(a+b+1)}}\leq p\leq1$
	\item If $	\frac{1}{2}+\frac{ab}{a+b}+\sqrt{\frac{1}{4}+\frac{a^2b^2}{(a+b)^2(a+b+1)}}\leq p\leq2$,  then $\mathcal{F}''_p$ is AM on $(0,1)$.
\end{enumerate}
\end{corollary}
\begin{remark}
Corollary \eqref{corollary1}(i) represents a significant improvement over the result in \cite[Theorem 1.1(i)]{WZ-IJPAM-2025}. The original work in  \cite{WZ-IJPAM-2025} suffered from two key limitations that restricted its applicability: it not only imposed an extraneous constraint $p\geq ab-1$ but also included a computationally cumbersome inequality that placed additional restrictions on the parameters $a$ and $b$. In contrast, the present study derives the exact necessary and sufficient condition $\frac{ab}{a+b}\leq p\leq1$ for the absolute monotonicity of $-\mathcal{F}'_p$ on $(0,1)$ with a natural constraint $ab\leq a+b$.
\end{remark}

\subsection{Absolute Monotonicity of $\mathcal{G}_p$}	

In this section, we establish a necessary and sufficient condition for $\mathcal{G}_p$ in the case $a+b\geq c$, and a sufficient condition for $-\mathcal{G}'_p$ in the zero-balanced case.
\begin{theorem}\label{theorem-Gp}{\bf (Case $c\leq a+b$)}\\
	Let $c\leq a+b$ and define $\mathcal{G}_p(x)=(1-x)^p\exp\left(F(a,b;c;x)\right)$ on $(0,1)$. Then the following statements hold:
	\begin{itemize}[leftmargin=2em]
		\item[(i)] $\mathcal{G}_p$ and $\mathcal{G}'_p$ are both AM on $(0,1)$ if and only if $p\leq\frac{ab}{c}$.  
		\item[(ii)] For $c=a+b$ with $a+b\geq 2ab(a+b+1)$,  $-\mathcal{G}'_p$ is AM on $(0,1)$ if $\frac{ab(2a+2b+1)}{(a+b)(a+b+1)}\leq p\leq1$.
	\end{itemize}
\end{theorem}

\begin{proof}
	(i) 	\textbf{\textsl{Necessity.}}  Suppose $\mathcal{G}_p$  and $\mathcal{G}'_p$ are both absolutely monotonic on $(0,1)$.  Evaluating the right-hand limit of $\mathcal{G}'_p(x)$ at $x=0^+$:
	\begin{equation*}
		\mathcal{G}'_p(0^+)=\lim_{x\to0^+}\frac{d}{dx}\left[(1-x)^p\exp(F(a,b;c;x))\right]=e\left(\frac{ab}{c}-p\right).
	\end{equation*}Absolute monotonicity of $\mathcal{G}'_p(x)$ requires $\mathcal{G}'_p(0^+)\geq0$, so  $p\leq \frac{ab}{c}$.
	
	\textbf{\textsl{Sufficiency.}}	We first consider the case $p \leq 0$. For this range, $(1-x)^p$ is absolutely monotonic (all its derivatives on $(0, 1)$ are non-negative), and $\exp\left(F(a,b;c;x)\right)$ is also absolutely monotonic (the derivative of the exponential function is itself, and $F(a,b;c;x) > 0$ on $(0, 1)$). The product of two absolutely monotonic functions is absolutely monotonic, so $\mathcal{G}_p$ and $\mathcal{G}_p'$ are absolutely monotonic for $p \leq 0$.
	
	For $p>0$, we proceed by induction on $k\in\mathbb{N}$ to show the result holds for all $p\leq \frac{ab}{c}$. 
    \begin{itemize}[leftmargin=1em]
    	\item {\bf Base Case ($0<p\leq \min\{1,\frac{ab}{c}\}$):}\\
    	Rewrite $\mathcal{G}_p(x)$ as a ratio of power series (consistent with Jurkat's criterion in Proposition \ref{prop-Jurkat}):
    	\begin{equation}\label{Gp-1}
    		\mathcal{G}_p(x)=\frac{\exp\left(F(a,b;c;x)\right)}{(1-x)^{-p}}=\frac{g_0(x)}{(1-x)^{-p}},
    	\end{equation}where
    	$g_0(x)=\exp(F(a,b;c;x))=\sum_{n=0}^\infty B_nx^n$ (with $B_n>0$ for all $n\geq0$) and $W_n=\frac{(p)_n}{n!}$. Construct the test function $\mathcal{T}_pg_0(x)$ (as in Proposition \ref{prop-method}) for $g_0(x)$:
    \begin{align}\nonumber
    		\mathcal{T}_pg_0(x)&=\left[\frac{ab}{c}(1-x)F(a+1,b+1;c+1;x)-p\right]\exp\left(F(a,b;c;x)\right)\\ \label{Tpg0}
    	&=\left[\left(\frac{ab}{c}-p\right)+\sum_{n=1}^{\infty}\frac{[ab+(a+b-c)n](a)_n(b)_n}{(c)_{n+1}n!}x^n\right]\exp\left(F(a,b;c;x)\right)
    \end{align}
    	Since $p\leq \frac{ab}{c}$ and $c\leq a+b$, $\mathcal{T}_pg_0(x)$  is the product of two absolutely monotonic functions-- hence absolutely monotonic itself.  By Proposition \ref{prop-method}, $\{B_n/W_n\}_{n\geq0}$ is increasing. By Lemma \ref{lemma-W_{n+1}/W_n}, for $0<p\leq1$, the sequence $\{W_{n+1}/W_n\}_{n\geq0}$ is increasing. Combining this with Jurkat's criterion (Proposition \ref{prop-Jurkat}) and \eqref{Gp-1},  $\mathcal{G}'_p$ is absolutely monotonic. Since $\mathcal{G}_p(0)=e>0$, $\mathcal{G}_p$ is also absolutely monotonic on $(0,1)$.
    	\item {\bf Inductive Step:}\\
    	Assume that for some $k\geq1$, $\mathcal{G}_p$ is absolutely monotonic on $(0,1)$ whenever $p\leq\min\{k,\frac{ab}{c}\}$. We now show the result extends to $k<p\leq \min\{k+1,\frac{ab}{c}\}$.
    	For this range, $k<\frac{ab}{c}$, so the inductive hypothesis implies $\mathcal{G}_k(x)=(1-x)^k\exp(F(a,b;c;x))$ is absolutely monotonic.
    	  Rewrite $\mathcal{G}_p$ as
    	  \begin{equation}\label{Fp-2}
    	  	\mathcal{G}_p(x)=\frac{(1-x)^k\exp\left(F(a,b;c;x)\right)}{(1-x)^{-(p-k)}}=\frac{g_k(x)}{\sum_{n=0}^\infty W^{^k}_nx^n},
    	  \end{equation}where $g_k(x)=\mathcal{G}_k(x)=\sum_{n=0}^{\infty}B^{^k}_nx^n$ and $W^{^k}_n=W_n(p-k)$ (with $W_n(\cdot)$ defined in \eqref{phi-Wn}).  Construct the test function $\mathcal{T}_{p-k}g_k(x)$ for $g_k(x)$:
    	  \begin{align*}
    	  \mathcal{T}_{p-k}g_k(x)&=\left[\frac{ab}{c}(1-x)F(a+1,b+1;c+1;x)-k-(p-k)\right]\mathcal{G}_k(x)\\
    	  &=\left[\left(\frac{ab}{c}-p\right)+\sum_{n=0}^{\infty}\frac{[ab+(a+b-c)n](a)_n(b)_n}{(c)_{n+1}n!}x^n\right]\mathcal{G}_k(x).
    	  \end{align*}
    	  Since $p\leq \frac{ab}{c}$ and $a+b\geq c$, the term in brackets is absolutely monotonic (as shown in the base case), and $\mathcal{G}_k(x)$ is absolutely monotonic by the inductive hypothesis. Thus, $\mathcal{T}_{p-k}g_k(x)$ is absolutely monotonic. By Proposition \ref{prop-method}, $\{B^{^k}_n/W^{^k}_n\}_{n\geq0}$ is increasing. Since $0<p-k\leq1$, Lemma \ref{lemma-W_{n+1}/W_n} implies $\{W^{^k}_{n+1}/W^{^k}_n\}_{n\geq0}$ is increasing. By Jurkat's criterion (Proposition \ref{prop-Jurkat}), $\mathcal{G}_p'$ is absolutely monotonic,and  $\mathcal{G}_p$ inherits absolutely monotonic (as $\mathcal{G}_p(0)=e>0$). By induction, $\mathcal{G}_p$ and $\mathcal{G}'_p$ are absolutely monotonic for all $p\leq \frac{ab}{c}$.
    \end{itemize}
	
	\medskip
	
	\noindent(ii) \textbf{Absolute Monotonicity of $-\mathcal{G}'_p(x)$ for $c=a+b$}
	
	For $c=a+b$ with $a+b\geq2ab(a+b+1)$, we first verify the interval $\left[\frac{ab(2a+2b+1)}{(a+b)(a+b+1)},1\right]$ is well-defined: 
	\begin{equation*}
		\frac{ab(2a+2b+1)}{(a+b)(a+b+1)}<\frac{2ab(a+b+1)}{a+b}\leq1,
	\end{equation*} 
	where the final inequality follows from $a+b\geq2ab(a+b+1)$.
	
	Rewrite $\mathcal{G}_p(x)$ as a power series ratio (analogous to \eqref{Gp-1}):
	\begin{equation}
		\mathcal{G}_p(x)=\frac{\exp(F(a,b;a+b;x))}{(1-x)^{-p}}=\frac{\hat{g}(x)}{\sum_{n=0}^{\infty}W_nx^n},
	\end{equation}
	where $\hat{g}(x)=\exp(F(a,b;c;x))=\sum_{n=0}^{\infty}\hat{B}_nx^n$ and $W_n=\frac{(p)_n}{n!}$. To analyze $-\mathcal{G}'_p(x)$, Proposition \ref{prop-method} dictates we study the test function $-\mathcal{T}_p\hat{g}(x)$:
	\begin{equation*}
		-\mathcal{T}_p\hat{g}(x)=\left[p-\frac{ab}{a+b}F(a,b;a+b+1)\right]\exp(F(a,b;a+b;x))=G(x)\cdot\mathcal{G}_{\frac{ab}{a+b}}(x),
	\end{equation*}
	where $G(x)=\left[p-\frac{ab}{a+b}F(a,b;a+b+1)\right](1-x)^{-\frac{ab}{a+b}}.$
	
	Express $G(x)$ as a power series: 
	\begin{equation*}
		G(x)=p(1-x)^{-\frac{ab}{a+b}}-\frac{ab}{a+b}\bar{g}(x)=\sum_{n=0}^\infty\left(p\bar{W}_n-\frac{ab}{a+b}\bar{B}_n\right)x^n,
	\end{equation*}with $\bar{g}(x)=(1-x)^{-\frac{ab}{a+b}}F(a,b;a+b+1;x)=\sum_{n=0}^\infty \bar{B}_nx^n$ and $\bar{W}_n=W_n(\frac{ab}{a+b})$. Next, consider the test function $\mathcal{T}_{\frac{ab}{a+b}}\bar{g}(x)$ for $\bar{g}(x)$:
	\begin{equation}\label{Tg}
		\mathcal{T}_{\frac{ab}{a+b}}\bar{g}(x)=\frac{ab}{a+b+1}(1-x)^{1-\frac{ab}{a+b}}F(a+1,b+1;a+b+2;x).
	\end{equation}
	From $a+b\geq2ab(a+b+1)$, we derive:
	\begin{equation*}
		1-\frac{ab}{a+b}-\frac{(a+1)(b+1)}{a+b+2}=\frac{a+b-2ab(a+b+1)}{(a+b)(a+b+2)}\geq0,
	\end{equation*}
	so  $\frac{(a+1)(b+1)}{a+b+2}\leq1-\frac{ab}{a+b}<1$. 
	Combining this with \eqref{Tg} and Theorem \ref{theorem-Fp}(i),  $-\left[\mathcal{T}_{\frac{ab}{a+b}}\bar{g}(x)\right]'$ is absolutely monotonic on $(0,1)$. Using \eqref{V_n/W_n} and $\mathcal{T}_{\frac{ab}{a+b}}\bar{g}(0)=\frac{ab}{a+b+1}>0$,  the sequence $\{\bar{B}_n/\bar{W}_n\}$ is increasing for $0\leq n\leq1$ and decreasing for $n\geq1$. So $\frac{\bar{B}_n}{\bar{W}_n}\leq \frac{\bar{B}_1}{\bar{W}_1}=\frac{2a+2b+1}{a+b+1}$
	for all $n\geq0$. For $\frac{ab(2a+2b+1)}{(a+b)(a+b+1)}\leq p\leq1$:  
	\begin{align*}
		p\bar{W}_n-\frac{ab}{a+b}\bar{B}_n&\geq\frac{ab(2a+2b+1)}{(a+b)(a+b+1)}\bar{W}_n-\frac{ab}{a+b}\bar{B}_n\\
		&=\frac{ab\bar{W}_n}{a+b}\left(\frac{2a+2b+1}{a+b+1}-\frac{\bar{B}_n}{\bar{W}_n}\right)\geq0
	\end{align*}Thus,  $G(x)$ is absolutely monotonic. By part (i), $\mathcal{G}_{\frac{ab}{a+b}}$ is also absolutely monotonic, so $-\mathcal{T}_p\hat{g}(x)$ (the product of two absolutely monotonic functions) is absolutely monotonic. By Proposition \ref{prop-method}, $\{\widehat{B}_n/W_n\}_{n\geq0}$ is decreasing; by Lemma \ref{lemma-W_{n+1}/W_n}, $\{W_{n+1}/W_n\}_{n\geq0}$ is increasing. Combining these with Jurkat's criterion (Proposition \ref{prop-Jurkat}),  $-\mathcal{G}'_p$ is absolutely monotonic on $(0,1)$.
\end{proof}

\begin{theorem}{\bf (Case $c>a+b$)}\label{theorem-Gp-2}\\
Let $c\geq a+b+ab$ and $\mathcal{G}_p(x)$ be defined as in Theorem \ref{theorem-Gp}. Then $-\mathcal{G}'_p$ is AM on $(0,1)$ if and only if $\frac{ab}{c}\leq p\leq1$.
\end{theorem}

\begin{proof}
\textbf{Necessity.} 
Suppose	$-\mathcal{G}'_p$ is absolutely monotonic, then $\mathcal{G}'_p(x)\leq0$ for all $x\in(0,1)$.
\begin{itemize}[leftmargin=2em]
	\item At $x\to0^+$: Evaluating the right-hand limit of $\mathcal{G}'_p(x)$ at $x=0^+$:
	\begin{equation*}
		\mathcal{G}'_p(0^+)=\lim_{x\to0^+}\frac{d}{dx}\left[(1-x)^p\exp(F(a,b;c;x))\right]=e\left(\frac{ab}{c}-p\right)\leq0,
	\end{equation*}which gives $p\geq\frac{ab}{c}$.
	\item At $x\to1^-$: For $c>a+b$, $F(a,b;c;1^-)$ is finite by \eqref{c>a+b}, and using the symmetry relation \eqref{sym-relation},  $(1-x)F(a+1,b+1;c+1;x)=(1-x)^{c-a-b}F(c-a,c-b;c+1;x)\to0$ (since  $c+1\geq(a+1)+(b+1)$ if $c\geq a+b+1$; $c+1>(c-a)+(c-b)$ if $a+b+ab<c<a+b+1$). Thus, if $p>1$,
	\begin{equation*}
		\mathcal{G}'_p(1^-)=\lim_{x\to1^-}(1-x)^{p-1}\left[\frac{ab}{c}(1-x)F(\cdot)-p\right]\exp(F(\cdot)),
	\end{equation*}which forces $\mathcal{G}'_p$ not to be absolutely monotonic. Thus, $p\leq1$ and collectively, $\frac{ab}{c}\leq p\leq1$.
\end{itemize}
	\noindent\textbf{Sufficiency.} Assume $\frac{ab}{c} \leq p \leq 1$, we use Jurkat’s criterion (Proposition \ref{prop-Jurkat}) and Proposition \ref{prop-method}.
    Following the form in \eqref{Gp-1}, it suffices to $-\mathcal{T}_pg_0(x)$ is absolutely monotonic.
    
    For $c\geq a+b+ab$:
	\begin{itemize}
		\item[$\circ$] $p-\frac{ab}{c}\geq0$ (constant term);
		\item[$\circ$] $(c-a-b)n-ab\geq c-a-b-ab\geq0$ for $n\geq1$ (series coefficients).
	\end{itemize}
	Thus, $-\mathcal{T}_pg_0(x)=\left[p-\frac{ab}{c}+\sum\text{non-neg. terms}\right]\exp(F(\cdot))$ (from \eqref{Tpg0})--a product of two absolutely monotonic functions, hence absolutely monotonic.
    
    Apply Jurkat's criterion:
    \begin{itemize}
    	\item[$\circ$] By Proposition \ref{prop-method}, $\{B_n/W_n\}_{n\geq0}$ is decreasing (since $-\mathcal{T}_pg_0(x)$ is absolutely monotonic);
    	\item[$\circ$] By Lemma \ref{lemma-W_{n+1}/W_n}, $\{W_{n+1}/W_n\}_{n\geq0}$ is increasing (since $p\leq1$).
    \end{itemize}
    By Proposition \ref{prop-Jurkat}, $-\mathcal{G}'_p=-\left(\frac{\sum B_nx^n}{\sum W_nx^n}\right)'$ is absolutely monotonic for $\frac{ab}{c}\leq p\leq1$.
\end{proof}

\subsection{Absolute Monotonicity of $\ln\mathcal{F}_p$}	
The absolute monotonicity of $\ln\mathcal{F}_p(x)$ is determined by the monotonicity of the Maclaurin coefficients of $\ln F(a,b;c;x)$.

\begin{theorem}\label{theorem-lnF}
	Let $\mathcal{F}_p(x)$ be defined in Theorem \ref{theorem-Fp}. and $C_n$ denote the Maclaurin coefficients of $\ln F(a,b;c;x)$ (i.e., $\ln F(a,b;c;x)=\sum_{n=1}^\infty C_nx^n$). Define two parameter regions:
	\begin{align*}
		\mathfrak{R}_1&=\{(a,b,c)\in\mathbb{R}_+^3|\ c\geq a+b-1,\ (1+a+b-ab)c\geq2ab,\ (c-a)(c-b)\geq0\},\\
			\mathfrak{R}_2&=\{(a,b,c)\in\mathbb{R}_+^3|\ c\geq a+b-1,\ (c-a)(c-b)\leq0\}.
	\end{align*}
	(i) For $(a,b,c)\in\mathfrak{R}_1$:
		\begin{itemize}[leftmargin=2em,itemindent=0em]
			\item[$\circ$]$\ln\mathcal{F}_p$ is AM on $(0,1)$ if and only if $p\leq\max\{0,a+b-c\}$;
			\item[$\circ$] $-\ln\mathcal{F}_p$ is AM on $(0,1)$ if and only if $p\geq\frac{ab}{c}$;
			\item[$\circ$] For $k\geq1$,  $\left(\ln\mathcal{F}_p\right)^{(k)}$ (resp. $-\left(\ln\mathcal{F}_p\right)^{(k)}$) is AM on $(0,1)$ if and only if $p\leq \max\{0,a+b-c\}$ (resp. $p\geq kC_k$);
		\end{itemize}
	(ii) For $(a,b,c)\in \mathfrak{R}_2$: 
		\begin{itemize}[leftmargin=2em,itemindent=0em]
			\item[$\circ$]$\ln\mathcal{F}_p$ is AM on $(0,1)$ if and only if $p\leq\frac{ab}{c}$;
			\item[$\circ$]$-\ln\mathcal{F}_p$ is AM on $(0,1)$ if and only if $p\geq a+b-c$;
			\item[$\circ$]For $k\geq1$, $\left(\ln\mathcal{F}_p\right)^{(k)}$ (resp. $-\left(\ln\mathcal{F}_p\right)^{(k)}$) is AM on $(0,1)$ if and only if $p\leq kC_k$ (resp. $p\geq a+b-c$).
		\end{itemize}
\end{theorem}	
\proof
	First, rewrite $\ln\mathcal{F}_p(x)$ using the Maclaurin expansion of $\ln F(a,b;c;x)$:
	\begin{equation}\label{H_p-expr}
		\ln\mathcal{F}_p(x)=p\ln(1-x)+\ln F(a,b;c;x)=\sum_{n=1}^\infty\frac{nC_n-p}{n}x^n.
	\end{equation}
	The absolute monotonicity of $\ln\mathcal{F}_p(x)$ (or its derivatives) depends on the non-negativity (or non-positivity) of the coeffcients $\frac{nC_n-p}{n}$, which is equivalent to the monotonicity of $\{n C_n\}_{n\geq1}$ and the limit of $nC_n$ as $n\to\infty$.
	
	\begin{description}[leftmargin=0.5em]
		\item[Step 1] Compute $\lim\limits_{n\to\infty}nC_n$.\\
		Differentiate $\ln F(a,b;c;x)$ and multiply by $(1-x)$:
		\begin{equation}\label{theorem5-H}
			\frac{ab}{c}\frac{(1-x)F(a+1,b+1;c+1;x)}{F(a,b;c;x)}=\sum_{n=0}^{\infty}\Big[(n+1)C_{n+1}-nC_n\Big]x^{n}.
		\end{equation}
		Taking the limit as $x\to1^-$ and using the symmetry relation \eqref{sym-relation} for the hypergeometric function, we obtain
		\begin{equation}\label{limnC_n}
			\lim_{n\to\infty}nC_n=\sum_{n=0}^\infty\Big[(n+1)C_{n+1}-nC_n\Big]=\lim_{x\to1^-}\frac{ab}{c}\frac{(1-x)F(a+1,b+1;c+1;x)}{F(a,b;c;x)}.
		\end{equation}
		By the asymptotic behavior of the hypergeometric function \eqref{c>a+b}-\eqref{c<a+b}:
		\begin{itemize}[leftmargin=1em]
			\item If $c>a+b$, then $\lim\limits_{x\to1^-}F(a+1,b+1;c+1,x)$ is finite (if $c\geq a+b+1$) and $\lim\limits_{x\to1^-}F(c-a,c-b;c+1,x)$ is also finite (if $a+b<c<a+b+1$), which implies $\lim\limits_{x\to1^-}\left[(1-x)F(a+1,b+1;c+1;x)\right]=0$. In this case, $\lim\limits_{x\to1^-}F(a,b;c;x)$ is finite and non-zero, by \eqref{limnC_n}, leading to $\lim\limits_{n\to\infty}nC_n=0$;
			\item If $c=a+b$, then $\lim\limits_{x\to1^-}\left[(1-x)F(a+1,b+1;c+1;x)\right]$ is finite and non-zero, while $\lim\limits_{x\to1^-}F(a,b;c;x)$ is infinite, which implies $\lim\limits_{n\to\infty}nC_n=0$;
			\item If $c<a+b$, then it follows from \eqref{c>a+b} that $\lim\limits_{x\to1^-}F(c-a,c-b;c+1;x)=\frac{\Gamma(c+1)\Gamma(a+b+1-c)}{\Gamma(a+1)\Gamma(b+1)}$ and $\lim\limits_{x\to1^-}F(c-a,c-b;c;x)=\frac{\Gamma(c)\Gamma(a+b-c)}{\Gamma(a)\Gamma(b)}$, which together with \eqref{limnC_n} gives $\lim\limits_{n\to\infty}nC_n=a+b-c$.
		\end{itemize}
		Collectively, $\lim\limits_{n\to\infty}nC_n=\max\{0,a+b-c\}$.
		\item[Step 2] Monotonicity of $\{nC_n\}_{n\geq1}$.\\
		From \eqref{theorem5-H} and Lemma \ref{lemma-H}, we see that $H(x)=\sum_{n=0}^\infty\left[(n+1)C_{n+1}-nC_n\right]x^n$, so the monotonicity of $\{nC_n\}_{n\geq1}$ is equivalent to the sign of the coefficients of $H'(x)$. 
		\begin{itemize}[leftmargin=1em]
			\item For $(a,b,c)\in\mathfrak{R}_1$: By Lemma \ref{lemma-H}, $-H'(x)$ is absolutely monotonic, so its coefficients are non-negative. This implies $(n+1)C_{n+1}-nC_n\leq0$ for all $n\geq0$, i.e., $\{nC_n\}_{n\geq1}$ is decreasing and thereby $\frac{ab}{c}=1\cdot C_1\geq nC_n\geq\lim\lim\limits_{n\to\infty}nC_n=\max\{0,a+b-c\}$.
			\item For $(a,b,c)\in\mathfrak{R}_2$:  We first claim that $(1+a+b-ab)c\geq2ab$. We may assume that $a\leq b$. It follows  from $(c-a)(c-b)\leq0$ that $a\leq c\leq b$ and thereby $a+b-1\leq c\leq b$, which gives $a\leq1$ and then $1+a+b-ab=1+a+(1-a)b>0$. Hence, $(1+a+b-ab)c-2ab\geq (1+a+b-ab)a-2ab=a(a+1)(1-b)\geq0$ if $b\leq1$ and $(1+a+b-ab)c-2ab\geq(1+a+b-ab)(a+b-1)-2ab=(1-a)(b-1)(1+a+b)\geq0$ if $b>1$. This gives the above claim. By Lemma \ref{lemma-H}, $H'(x)$ is absolutely monotonic and so $\{nC_n\}_{n\geq1}$ is increasing. Consequently, $\frac{ab}{c}=1\cdot C_1=nC_n\leq \lim\limits_{n\to\infty}nC_n=a+b-c$ (since $c<a+b$).
		\end{itemize}
		\item[Step 3] Derive the necessary and sufficient conditions.\\
		For $(a,b,c)\in\mathfrak{R}_1$:
		\begin{itemize}[leftmargin=1.5em]
			\item[$\circ$] $\ln\mathcal{F}_p$ is absolute monotonic if and only if $nC_n-p\geq0$ for all $n\geq1$, equivalently, $p\leq nC_n$ for all $n\geq1$. That is to say, $p\leq \min_{n\geq1}\{nC_n\}=\max\{0,a+b-c\}$;
			\item[$\circ$] $-\ln\mathcal{F}_p$ is absolute monotonic if and only if $nC_n-p\leq0$ for all $n\geq1$, equivalently, $p\geq nC_n$ for all $n\geq1$, if and only if $p\geq\max_{n\geq1}\{nC_n\}=\frac{ab}{c}$;
			\item[$\circ$] For $k\geq1$, $-(\ln\mathcal{F}_p)^{(k)}$ (resp. $(\ln\mathcal{F}_p)^{(k)}$) is absolutely monotonic if and only if the coefficients of this Maclaurin series are non-negative, which is equivalent to $p\geq\max_{n\geq k}\{nC_n\}=kC_k$ (resp. $p\leq\min_{n\geq k}\{nC_n\}=\max\{0,a+b-c\}$), since $\{nC_n\}_{n\geq1}$ is decreasing.
		\end{itemize}
		For $(a,b,c)\in\mathfrak{R}_2$:
		\begin{itemize}[leftmargin=1.5em]
			\item[$\circ$] $\ln\mathcal{F}_p$ is absolute monotonic if and only if $nC_n-p\geq0$ for all $n\geq1$, i.e., $p\leq nC_n$ for all $n\geq1$. That is to say, $p\leq \min_{n\geq1}\{nC_n\}=\frac{ab}{c}$;
			\item[$\circ$] $-\ln\mathcal{F}_p$ is absolute monotonic iff $nC_n-p\leq0$ for all $n\geq1$, i.e., $p\geq nC_n$ for all $n\geq1$, iff $p\geq\max_{n\geq1}\{nC_n\}=a+b-c$.
			\item[$\circ$] For $k\geq1$, $(\ln\mathcal{F}_p)^{(k)}$ (resp. $-(\ln\mathcal{F}_p)^{(k)}$ )is absolutely monotonic if and only if the coefficients of this Maclaurin series are non-negative. Direct computation shows this is equivalent to $p\leq\min_{n\geq k}\{nC_n\}=kC_k$ (resp. $p\leq\max_{n\geq k}\{nC_n\}=a+b-c$), since $\{nC_n\}_{n\geq1}$ is increasing.
		\end{itemize}
This completes the proof. \qed
	\end{description}

\begin{remark}
	Theorem \ref{theorem-lnF} offers a streamlined and novel proof for establishing the necessary and sufficient conditions that ensure $\pm\ln\mathcal{F}_p(x)$ and their higher-order derivatives are absolutely monotonic within specific parameter regions. A critical advancement here lies in the refinement of parameter domains compared to existing literature \cite{WZ-BIMS-2024}. Specifically, the parameter region $\mathfrak{R}_2$ in this work aligns with the region $\mathcal{R}_2$ defined in \cite{WZ-BIMS-2024}, while $\mathfrak{R}_1$ represents a strict improvement over the region $\mathcal{R}_1$ from the same reference.

    In \cite{WZ-BIMS-2024}, the region $\mathcal{R}_1$ imposes the restrictive constraint $c\geq a+b$; in contrast, we relax this condition to $c \geq a+b-1$ for $\mathfrak{R}_1$. This relaxation not only expands the scope of valid parameters but also leads to a new upper bound for $p$ that guarantees the absolute monotonicity of $\ln\mathcal{F}_p(x)$—a result that was unattainable under the stricter constraint of $\mathcal{R}_1$. Notably, the technical efficiency of our approach stands out: we establish the monotonicity of the sequence $\{nC_n\}_{n\geq1}$ and compute its limit as $n\to\infty$ with minimal exposition, which stands in sharp contrast to \cite{WZ-BIMS-2024}.
\end{remark}

\section{Inequalities involving the Gaussian hypergeometric function}
In this section, we utilize the absolute monotonicity of three function families  $\mathcal{F}_p$, $\mathcal{G}_p$ and $\ln\mathcal{F}_p$ (established in Section 3) to derive several new inequalities for the Gaussian hypergeometric function $F(a,b;c;x)$. 
	
\subsection{A Rational Approximation}By using the absolute monotonicity results of $\mathcal{F}_p(x)$ from Theorems \ref{theorem-Fp} and \ref{theorem-Fp-2}, we first establish a rational approximation for $F(a,b;c;x)$.
	
\begin{proposition}\label{prop-rational}
Let $u_n=u_n(p)$ denote the Maclaurin coefficients of $\mathcal{F}_p$ (i.e. $\mathcal{F}_p(x)=\sum_{n=0}^{\infty}u_n(p)x^n$). Let $p_*$ (smaller) and $p^*$ (larger) be two positive roots of $\tau(p)$ defined in Lemma \ref{lemma-tau}. Then the following statements hold:
\begin{itemize}[leftmargin=1em]
	\item For $c\geq a+b$:
	\medskip
	\begin{itemize}
		\item[$\circ$] If $\frac{ab}{c}\leq p\leq1$,  the double inequality
		\begin{equation}\label{R-ineq-1}
			\frac{\sum_{j=0}^nu_jx^j-\left(\sum_{k=0}^nu_k\right)x^{n+1}}{(1-x)^p}<F(a,b;c;x)<\frac{\sum_{j=0}^{n+1}u_jx^j}{(1-x)^p}
		\end{equation}holds for all $x\in(0,1)$ with $n\geq0$;
	\item[$\circ$] If $p_*\leq p\leq1$,  the double inequality above (Eq. \eqref{R-ineq-1}) holds for all $x\in(0,1)$ with $n\geq1$;
	\item[$\circ$] If $p^*\leq p\leq 2$, the double inequality
	\begin{equation}\label{R-ineq-2}
	\frac{\sum_{j=0}^{n+1}u_jx^j}{(1-x)^p}<F(a,b;c;x)<\frac{\sum_{j=0}^nu_jx^j-\left(\sum_{j=0}^nu_j\right)x^{n+1}}{(1-x)^p}
	\end{equation}holds for all $x\in(0,1)$ with $n\geq1$.
	\end{itemize}
	\item For $\max\{a,b\}<c<a+b$:
	\medskip
	\begin{itemize}[itemsep=0.45em]
		\item[$\circ$]  If $\frac{ab}{c}\leq p\leq a+b+1-c$, the inequality \eqref{R-ineq-1} holds for all $x\in(0,1)$ with $n\geq0$;
		\item[$\circ$]  If $p_*\leq p\leq a+b+1-c$, the  inequality \eqref{R-ineq-1} holds for all $x\in(0,1)$ with $n\geq1$;
		\item[$\circ$]  If $p^*\leq p\leq a+b+2-c$, the inequality \eqref{R-ineq-2} holds for all $x\in(0,1)$ with $n\geq1$.
	\end{itemize}
\end{itemize}
\end{proposition}	

\begin{proof}
We only prove inequality  \eqref{R-ineq-1}; the other inequalities can be derived similarly.

If $\frac{ab}{c}\leq p\leq1$, Theorem \ref{theorem-Fp}(i) shows that $-\mathcal{F}'_p$ is absolutely monotonic. This implies $u_n\leq0$ for all $n\geq1$, so the function
\begin{equation*}
\frac{1}{x^{n+1}}\left[\mathcal{F}_p(x)-\sum_{j=0}^nu_jx^j\right]=\sum_{j=0}^{\infty}u_{j+n+1}x^j
\end{equation*}is decreasing on $(0,1)$ for all $n\geq0$ and takes values in the interval $\left[-\sum_{j=0}^nu_j,u_{n+1}\right]$. In other words, the double inequality
\begin{equation*}
	-\sum_{j=0}^nu_j<\frac{1}{x^{n+1}}\left[\mathcal{F}_p(x)-\sum_{j=0}^nu_jx^j\right]<u_{n+1}
\end{equation*}holds for all $x\in(0,1)$, which is equivalent to inequality \eqref{R-ineq-1}.
\end{proof}	
	
\subsection{A Logarithmic Approximation}	
Next,  we apply the absolutely monotonicity of $\mathcal{G}_p(x)$ (from Theorems \ref{theorem-Gp} and \ref{theorem-Gp-2}) to establish a logarithmic approximation for $F(a,b;c;x)$.	
	
\begin{proposition}
	Let $v_n=v_n(p)$ denote the Maclaurin coefficients of $\mathcal{G}_p$ (i.e. $\mathcal{G}_p(x)=\sum_{n=0}^{\infty}v_n(p)x^n$). Then the following statements hold:
	\begin{itemize}[leftmargin=1em]
		\item For $c\leq a+b$:
		\medskip
		\begin{itemize}
			\item[$\circ$] If $p\leq \frac{ab}{c}$,  the following inequality
			\begin{equation}\label{log-ineq-1}
			F(a,b;c;x)>\ln\left[\frac{\sum_{k=0}^{n}v_jx^j}{(1-x)^p}\right]
			\end{equation} holds for all $x\in(0,1)$ with $n\geq0$;
			\item[$\circ$] In particular, for $c=a+b$ with $a+b\geq2ab(a+b+1)$, the reverse of inequality \eqref{log-ineq-1} holds for all $x\in(0,1)$ and $n\geq1$ if $\frac{ab(2a+2b+1)}{a+b}\leq p\leq1$.
		\end{itemize}
		\item For $c\geq a+b+ab$:
		\medskip
		\begin{itemize}
			\item[$\circ$]  If $\frac{ab}{c}\leq p\leq1$, the double inequality 
			\begin{equation}\label{log-ineq-2}
				\ln\left[\frac{\sum_{j=0}^nv_jx^j-\left(\sum_{j=0}^nv_j\right)x^{n+1}}{(1-x)^p}\right]<F(a,b;c;x)<\ln\left[\frac{\sum_{j=0}^{n+1}v_jx^j}{(1-x)^p}\right]
			\end{equation}holds for all $x\in(0,1)$ and $n\geq1$.
		\end{itemize}
	\end{itemize}
\end{proposition}	
\begin{proof}
For $c\leq a+b$:  If $p\leq\frac{ab}{c}$, Theorem \ref{theorem-Gp}(i) implies $\mathcal{G}_p(x)$ is absolutely monotonic, so $v_n \geq 0$ for all $n \geq 0$. Thus, 
\begin{equation*}
\mathcal{G}_p(x)-\sum_{j=0}^nv_kx^j=\sum_{j=n+1}^{\infty}v_jx^j>0
\end{equation*}for all $x\in(0,1)$ and $n\geq0$. Taking the natural logarithm of both sides yields inequality\eqref{log-ineq-1}. 

For $c\geq a+b+ab$: The function $\exp(F(a,b;c;x))$ is finite as $x\to1^-$, so $\mathcal{G}_p(1^-)=0$ when $\frac{ab}{c}\leq p\leq1$. Following the same reasoning as in the proof of Proposition \ref{prop-rational}, we obtain inequality \eqref{log-ineq-2}.
\end{proof}	
\subsection{A Exponential Approximation} 
Using the absolute monotonicity of $\ln \mathcal{F}_p(x)$ (from Theorem \ref{theorem-lnF}), we derive an exponential approximation for $F(a, b; c; x)$.	
	\begin{proposition}\label{prop-exp}
		Let $w_n=w_n(p)$ denote the Maclaurin coefficients of $\ln\mathcal{F}_p$ (i.e. $\ln\mathcal{F}_p(x)=\sum_{n=1}^{\infty}w_n(p)x^n$) and let $C_n$ denote the Maclaurin coefficients of $\ln F(a,b;c;x)$ (as defined in Theorem \ref{theorem-lnF}). Then the following statements hold:
		\begin{itemize}[leftmargin=1em]
			\item For $(a,b,c)\in\mathfrak{R}_1$:
			\medskip
			\begin{itemize}
				\item[$\circ$] If $p\leq\max\{0,a+b-c\}$ and $q\geq\frac{ab}{c}$,  the double inequality
				\begin{equation}\label{exp-ineq-1}
					\frac{\exp\left(\sum_{j=1}^{n}w_jx^j\right)}{(1-x)^p}<F(a,b;c;x)<\frac{\exp\left(\sum_{j=1}^{n}w_jx^j\right)}{(1-x)^q}
				\end{equation}holds for all $x\in(0,1)$ and $n\geq1$;
				\item[$\circ$] In particular, if $p\leq\max\{0,a+b-c\}$ and $q\geq kC_k$ for some $k\geq1$, the inequality \eqref{exp-ineq-1} holds for all $x\in(0,1)$ and $n\geq k$.
			\end{itemize}
			\item For $(a,b,c)\in\mathfrak{R}_2$:
			\medskip
			\begin{itemize}
				\item[$\circ$]  If $p\leq\frac{ab}{c}$ and $p\geq a+b-c$, the inequality \eqref{exp-ineq-1} holds for all $x\in(0,1)$ and $n\geq 1$;
				\item[$\circ$] In particular, if $p\leq kC_k$ and $q\geq a+b-c$ for some $k\geq1$, the inequality \eqref{exp-ineq-1} holds for all $x\in(0,1)$ and $n\geq k$.
			\end{itemize}
		\end{itemize}
	\end{proposition}	
	
	\begin{proof}
	For $(a,b,c)\in\mathfrak{R}_1$: 
	\begin{itemize}[leftmargin=2.5em]
		\item[$\circ$]If $p\leq\max\{0,a+b-c\}$,
		Theorem \ref{theorem-lnF} implies $\ln \mathcal{F}_p(x)$ is absolutely monotonic, so $	\ln\mathcal{F}_p(x)-\sum_{j=1}^{n}w_jx^j=\sum_{j=n+1}^{\infty}w_jx^j>0$ for all $x\in(0,1)$ and $n\geq1$. Exponentiating both sides gives the left-hand inequality of \eqref{exp-ineq-1}.
		\item[$\circ$] If $q\geq\frac{ab}{c}$, Theorem \ref{theorem-lnF} implies $-\ln \mathcal{F}_q(x)$ is absolutely monotonic, so $\ln\mathcal{F}_q(x)-\sum_{j=1}^{n}w_jx^j=\sum_{j=n+1}^{\infty}w_jx^j<0$ for all \(x \in (0, 1)\) and \(n \geq 1\). Exponentiating both sides gives the right-hand inequality of \eqref{exp-ineq-1}. 
	\end{itemize}
The remaining cases for $(a, b, c) \in \mathfrak{R}_2$ follow similarly from Theorem \ref{theorem-lnF}.
	\end{proof}
	\subsection{Inequalities for $F(a,b;c;r^p)/F(a,b;c;r^{p/q})$}
	We use Theorem \ref{theorem-lnF} to derive sharp upper and lower bounds for $F(a,b;c;r^p)/F(a,b;c;r^{p/q})$, where $p,q\in(1,\infty)$. This result improves the inequalities \cite{WZR-IJPAM-2023} and extends the applicability of $p$ and $q$ (unlike \cite{WZ-BIMS-2024}, which restricts $p$ and $p/q$ to integers).
	
	\textsl{Note}: To align with notation in \cite{WZ-BIMS-2024,WZR-IJPAM-2023}, $p$ and $q$ here denote real numbers greater than $1$ (not the parameters of $\ln \mathcal{F}_p(x)$ in previous sections).
	
	\begin{proposition}
	Let $p,q\in(1,\infty)$ and  define $\mathcal{Q}_{s,n}$ on $(0,1)$ by
	\begin{equation*}
		\mathcal{Q}_{s,n}(x)=\frac{1}{x^{1/q}-x}\ln\left[\frac{(1-x)^sF(a,b;c;x)}{(1-x^{1/q})^sF(a,b;c;x^{1/q})}\right]-\sum_{j=1}^{n}w_j(s)\frac{x^j-x^{j/q}}{x^{1/q}-x}
	\end{equation*}for $n\geq1$, where $w_j(s)$ denote the Maclaurin coefficients of $\ln\mathcal{F}_s$ (i.e. $\ln\mathcal{F}_s(x)=\sum_{j=1}^{\infty}w_j(s)x^j$).
	Then the following statements hold:
	\begin{itemize}[leftmargin=1em]
		\item[$\circ$] For $(a,b,c)\in\mathfrak{R}_1$: $\mathcal{Q}_{s,n}(x)<0$ for all $x\in(0,1)$ if $s\leq\max\{0,a+b-c\}$ and $\mathcal{Q}_{s,n}(x)>0$ for all $x\in(0,1)$ if $s\geq\frac{ab}{c}$. Consequently, the double inequality
	     \begin{equation}\label{ineq-prop6-1}
	     \begin{split}
	     		&\left(\frac{1-r^{p/q}}{1-r^p}\right)^{\frac{ab}{c}}\exp\left[\frac{ab(c-a)(c-b)(r^{2p/q}-r^{2p})}{2c^2(c+1)}\right]<\frac{F(a,b;c;r^p)}{F(a,b;c;r^{p/q})}\\
	     	&\quad <\left(\frac{1-r^{p/q}}{1-r^p}\right)^{s_0}\exp\left[\left(s_0-\frac{ab}{c}\right)(r^{p/q}-r^p)\right]
	     \end{split}
	     \end{equation}
	 holds for all $r\in(0,1)$, where $s_0=\max\{0,a+b-c\}$.
		\item[$\circ$] For $(a,b,c)\in\mathfrak{R}_2$: $\mathcal{Q}_{s,n}(x)<0$ for all $x\in(0,1)$ if $s\leq\frac{ab}{c}$ and $\mathcal{Q}_{s,n}(x)>0$ for all $x\in(0,1)$ if $s\geq a+b-c$. Consequently, the double inequality
		\begin{equation}\label{ineq-prop6-2}
			\begin{split}
				&\left(\frac{1-r^{p/q}}{1-r^p}\right)^{a+b-c}\exp\left[\frac{(a-c)(c-b)(r^{p/q}-r^p)}{c}\right]<\frac{F(a,b;c;r^p)}{F(a,b;c;r^{p/q})} \\
				&\quad <\left(\frac{1-r^{p/q}}{1-r^p}\right)^{\frac{ab}{c}}\exp\left[-\frac{ab(a-c)(c-b)(r^{2p/q}-r^{2p})}{2c^2(c+1)}\right]
			\end{split}
		\end{equation}holds for all $r\in(0,1)$
	\end{itemize}
	\end{proposition}
	\begin{proof}
	For $n\geq1$, rewrite $\mathcal{Q}_{s,n}(x)$ using the Maclaurin expansion of $\ln\mathcal{F}_s(x)$:
	\begin{align*}	\mathcal{Q}_{s,n}(x)&=\frac{\ln\mathcal{F}_s(x)-\ln\mathcal{F}_s(x^{1/q})}{x^{1/q}-x}-\sum_{j=1}^{n}w_j(s)\frac{x^j-x^{j/q}}{x^{1/q}-x}\\
		&=-\sum_{j=n+1}^{\infty}w_j(s)\frac{x^{j/q}-x^j}{x^{1/q}-x}.
	\end{align*}
	Since $q>1$, we have $x^{1/q}-x>0$ and  $x^{j/q}-x^j>0$ for all $x\in(0,1)$ and integers $j\geq1$.
	\begin{itemize}[leftmargin=1em]
		\item[$\circ$] For $(a,b,c)\in\mathfrak{R}_1$:  If $s\leq\max\{0,a+b-c\}$, Theorem \ref{theorem-lnF} implies $w_j(s)\geq0$ for $j\geq1$, so $\mathcal{Q}_{s,n}(x)<0$. If $s\geq\frac{ab}{c}$, Theorem \ref{theorem-lnF} implies $w_j(s)\leq0$ for $j\geq1$, so $\mathcal{Q}_{s,n}(x)>0$. Setting $x=r^p$ and simplifying yields  \eqref{ineq-prop6-1}.
	\item[$\circ$] For $(a,b,c)\in\mathfrak{R}_2$:  If $s\leq\frac{ab}{c}$, Theorem \ref{theorem-lnF} implies $w_j(s)\geq0$ for $j\geq1$, so $\mathcal{Q}_{s,n}(x)<0$. If $s\geq a+b-c$, Theorem \ref{theorem-lnF} implies $w_j(s)\leq0$, so $\mathcal{Q}_{s,n}(x)>0$. Setting $x=r^p$ and simplifying yields \eqref{ineq-prop6-2}.	
	\end{itemize}
	This completes the proof.
	\end{proof}
	
	\begin{remark}
	For $(a,b,c)\in\mathfrak{R}_1$,  $(c-a)(c-b)>0$, which implies
	$s_0-\frac{ab}{c}<0$ (either $s_0=0$ for $c\geq a+b$) or $s_0-\frac{ab}{c}=-\frac{(c-a)(c-b)}{c}<0$ for $c<a+b$).
	This shows inequality \eqref{ineq-prop6-1} strictly improves \cite[Eqs. (2.2)-(2.3)]{WZR-IJPAM-2023}. For $(a,b,c)\in\mathfrak{R}_2$, $(a-c)(c-b)>0$, so inequality \eqref{ineq-prop6-2} improves \cite[Eq. (2.4)]{WZR-IJPAM-2023}.
	\end{remark}

	\bigskip
	\begin{small}
	\begin{ack}
	We would like to thank the anonymous referees for their careful review of the manuscript.
	\end{ack}	
		
	\noindent{\bf Funding.}\ \ 	This research was supported by the National Natural Science Foundation of China (11971142) and the Natural Science Foundation of Zhejiang Province (LY19A010012).
		
	\medskip
		
		\noindent{\bf Data availability}\ \ Not applicable.
	\medskip
	
	\noindent{\bf\small Conflict of interest}\ \ \small{The author declares no conflict of interest.}
		\end{small}
	
\end{document}